\documentclass[11pt,reqno]{amsart}
\usepackage{CJK}
\usepackage{amssymb}
\usepackage{amsmath}
\usepackage{amsthm}
\usepackage{bm}

\usepackage{graphicx}

\usepackage[text={6.5in,9in},centering,letterpaper,dvips]{geometry}
\usepackage{array}              

\newtheorem{Prop}{Proposition}[section]
\newtheorem{Thm}[Prop]{Theorem}
\newtheorem{Lem}[Prop]{Lemma}
\newtheorem{Cor}[Prop]{Corollary}

\newtheorem*{Thm0}{Theorem}

\theoremstyle{definition}
\newtheorem{Ex}[Prop]{Example}

\begin{document}

\title{A NOTE ON THE CONCORDANCE INVARIANT EPSILON}

\author{Shida Wang}

\email{wang217@indiana.edu}

\begin{abstract}

We compare the smooth concordance invariants Upsilon, phi and epsilon.
Previous work gave examples of knots with one of the Upsilon and phi invariants zero but the epsilon invariant nonzero.
We build an infinite family of linearly independent knots with both the Upsilon and phi invariants zero but the epsilon invariant nonzero.

\end{abstract}

\maketitle

\section{Introduction}

There are three recently defined smooth concordance invariants derived from knot Heegaard Floer theory~\cite{CFKdef1,CFKdef2}:
the~$\varepsilon$-invariant~\cite{epsilon} defined by Hom,
the~$\Upsilon$-invariant~\cite{upsilon} defined by Ozsv\'{a}th-Stipsicz-Szab\'{o},
and the~$\varphi$-invariant~\cite{phi} defined by Dai-Hom-Stoffregen-Truong.
These invariants have been demonstrated to be very useful in proving linear independence of families of knots in the smooth concordance group~$\mathcal{C}$.

One may wonder whether one of the three invariants is stronger than the others.
In fact,\linebreak the~$\varphi$-invariant factors through equivalence classes induced by the~$\varepsilon$-invariant,
so the vanishing of the~$\varepsilon$-invariant implies that of the~$\varphi$-invariant.

In~\cite{example0}, Hom gives a knot with vanishing~$\Upsilon$-invariant but nonvanishing~$\varepsilon$-invariant.
This knot is also shown to have nonvanishing~$\varphi$-invariant~\cite[Proposition 1.10]{phi}.
Hence the~$\varphi$-invariant is not weaker than the~$\Upsilon$-invariant.
Furthermore, the author provides infinitely many linearly independent such knots~\cite{example1,example2}.

It is not known yet whether there is a knot with vanishing~$\varepsilon$-invariant but nonvanishing~$\Upsilon$\nobreakdash-invariant.
In~\cite{example2}, the author gives infinite families of linearly independent knots with vanishing~$\varphi$\nobreakdash-invariant but nonvanishing $\Upsilon$-invariant.
These knots are proved to have nonvanishing~$\varepsilon$\nobreakdash-invariant.

In this note, we will establish the following result.

\begin{Thm0}There exists a subgroup of $\mathcal{C}$ isomorphic to $\mathbb{Z}^\infty$ such that each of its nonzero elements has vanishing~$\Upsilon$-invariant and~$\varphi$-invariant but nonvanishing $\varepsilon$-invariant.\end{Thm0}

Summarizing, we have the following table.

\setlength{\doublerulesep}{0pt}
\begin{center}
\begin{tabular}{||c||c|c||}
  \hline\hline
   & $\Upsilon=0$ & $\Upsilon\neq0$ \\  \hline\hline
  $\varepsilon=0$ & easy (e.g. unknot and $T_{2,5}\#-T_{2,3}\#-T_{2,3}$) & open question \\  \hline
  $\varepsilon\neq0$ and $\varphi=0$ & this paper & \cite{example2} \\  \hline
  $\varphi\neq0$ & \cite{example0,phi,example1,example2} & easy (many) \\
  \hline\hline
\end{tabular}
\end{center}

Feller-Krcatovich~\cite{recursive} and the author~\cite{example2} proved that the~$\Upsilon$-invariant and~$\varphi$-invariant of torus knots satisfy certain recursive formulas
(see Propositions~\ref{Upsilon recursive} and~\ref{phi recursive} for details);
using these, knots with vanishing~$\Upsilon$-invariant or~$\varphi$-invariant can be constructed from connected sums of torus knots with suitable parameters.
This approach was used in~\cite{example1} and~\cite{example2} to provide desired knots.
Those knots have only one of $\Upsilon$ and~$\varphi$ zero, because the recursive formulas for the two invariants are different.
One may try to build a connected sum of torus knots that simultaneously fits the two recursive relations,
but for which there is a good chance to\linebreak obtain a knot stably equivalent~\cite{survey} to the unknot (e.g. $T_{2,5}\#-T_{2,3}\#-T_{2,3}$) and hence has vanishing~$\varepsilon$-invariant.
In our construction here, we carefully choose parameters of torus knots and manage to produce a family of connected sums that meet the requirements in the theorem.

\emph{Acknowledgments.} The author wishes to express sincere thanks to Professor Charles Livingston for carefully reading a draft of this paper
and detailed suggestions on grammar.

\section{Preliminaries}

We assume the reader is familiar with the knot Floer complex, the~$\varepsilon$-invariant, the~$\Upsilon$-invariant and the~$\varphi$-invariant.
For a knot $K\subset S^3$, these are denoted by~$\mathit{CFK}^\infty(K)$,~$\varepsilon(K)$,~$\Upsilon_K(t)$\linebreak and~$\varphi(K)=(\varphi_j(K))_{j=1}^\infty$, respectively.
The chain complex $\mathit{CFK}^\infty(K)$ is doubly filtered, free and finitely generated over~$\mathbb{Z}_2[U,U^{-1}]$.
Up to filtered chain homotopy equivalence, this complex is an invariant of~$K$.

We are considering three smooth concordance invariants.
For each knot~$K$, the invariant~$\varepsilon(K)$ is an integer in the set~$\{-1,0,1\}$;
the invariant~$\Upsilon_K(t)$ is a piecewise linear function on~$[0,2]$;
and~$\varphi(K)$ is a sequence of integers with at most finitely many nonzero terms.
Moreover,~$\Upsilon_K(t)$ and~$\varphi(K)$ give homomorphisms from the smooth concordance group~$\mathcal{C}$.

\subsection{Recursive formulas}

Now we state two facts that will produce many knots with\linebreak vanishing~$\Upsilon$-invariant or~$\varphi$-invariant.

\begin{Prop}\label{Upsilon recursive}\emph{(\cite[Proposition 2.2]{recursive})} Suppose $p$ and $r$ are relatively prime positive integers and $k$ is a nonnegative integer.
Then $\Upsilon_{T_{p,kp+r}}(t)=\Upsilon_{T_{r,p}}(t)+k\Upsilon_{T_{p,p+1}}(t)$.\end{Prop}

\begin{Prop}\label{phi recursive}\emph{(\cite[Theorem 1.4]{example2})} Suppose $p$ and $r$ are relatively prime positive integers with~$r<p$ and~$k$ is a nonnegative integer.
Then $$\varphi(T_{p,kp+r})=(k+1)\varphi(T_{r,p})+k\varphi(T_{p-r,p})+k(\varphi(T_{p,p+1})-\varphi(T_{p-1,p})).$$\end{Prop}

Thus any knot of the form $T_{p,kp+r}\#-T_{r,p}\#-kT_{p,p+1}$ has vanishing~$\Upsilon$-invariant,
where~$-K$ means the mirror image of the knot~$K$ with reversed orientation
(representing the inverse element of~$K$ in~$\mathcal{C}$) and~$kK$ means the connected sum of~$k$ copies of~$K$.
Similarly, any knot of the\linebreak form $T_{p,kp+r}\#-(k+1)T_{r,p}\#-kT_{r,p}\#-k(T_{p,p+1}\#-T_{p-1,p})$ has vanishing~$\varphi$-invariant.

\subsection{Conventions and Notations}

The definition of the~$\varepsilon$-invariant can be extended to a class~$\mathfrak{C}$ of chain complexes satisfying certain homological conditions (see~\cite[Definition~2.2]{1nn1}),
so that~$\varepsilon(K)=\varepsilon(\mathit{CFK}^\infty(K))$ for all knots~$K$.
The relation~$\sim_\varepsilon$ defined by~$C\sim_\varepsilon C'\Leftrightarrow\varepsilon(C\otimes C'^*)=0$ is an equivalence relation on~$\mathfrak{C}$,
making $\mathcal{CFK}_\mathrm{alg}:=\mathfrak{C}/\sim_\varepsilon$ an abelian group with the tensor product operation as the addition and the dual as the negative.
The binary relation defined by~$[C]>[C']\Leftrightarrow\varepsilon(C\otimes C'^*)=1$ gives a total order on~$\mathcal{CFK}_\mathrm{alg}$
that respects the addition operation~\cite[Proposition 4.1]{ordered},
where the pair of brackets denotes the~$\varepsilon$-equivalence class of the complex.
Hence~$\mathcal{CFK}_\mathrm{alg}$ becomes a \emph{totally ordered abelian group}.
It includes a\linebreak subgroup~$\mathcal{CFK}:=\{[\mathit{CFK}^\infty(K)]\mid K\text{ is a knot}\}$,
which is a quotient group of~$\mathcal{C}$ with the quotient map~$K\mapsto[\mathit{CFK}^\infty(K)]$.

For two elements $g,h\geqslant0$ of a totally ordered abelian group, where~$0$ is the identity element, we write~$g\ll h$ if~$n\cdot g<h$ for any natural number~$n$.

Given an even number of positive integers~$b_1,\cdots,b_{2m}$ that satisfy $b_i=b_{2m+1-i}$ for~$i=1,\cdots,2m$,
one can define a \emph{staircase complex}~$\mathrm{St}(b_1,\cdots,b_{2m})$ in $\mathfrak{C}$.
It is freely generated over~$\mathbb{Z}_2[U,U^{-1}]$ by a basis $x_0,x_1,\cdots,x_{2m}$
with the filtration level of $x_{2i}$ being~$(\sum_{j=1}^ib_{2j-1},\sum_{j=i+1}^mb_{2j})$\linebreak for $i=0,1,\cdots,m$
and the filtration level of $x_{2i-1}$ being~$(\sum_{j=1}^ib_{2j-1},\sum_{j=i}^mb_{2j})$ for $i=1,\cdots,m$.
The differential is defined by $\partial x_i=x_{i-1}+x_{i+1}$ for odd~$i$ and~$\partial x_i=0$ for even~$i$.
The\linebreak $\varepsilon$-equivalence class of the complex~$\mathrm{St}(b_1,\cdots,b_{2m})$ is denoted by~$[b_1,\cdots,b_m]$.
Note that we adopt the notation used in~\cite{filtration} rather than in~\cite{1nn1,example1,example2}.
For example,~$[\mathrm{St}(1,2,2,1)]$ is denoted by~$[1,2]$, not by~$[1,2,2,1]$.
We abbreviate the finite sequence $\underbrace{a_1,\cdots,a_n,\cdots,a_1,\cdots,a_n}_{k\text{ copies of }a_1,\cdots,a_n}$ to~$(a_1,\cdots,a_n)^k$.
For instance,~$[(1,2)^3]=[1,2,1,2,1,2]=[\mathrm{St}((1,2)^3,(2,1)^3)]$.
The~$\varepsilon$-equivalence class of any staircase complex is positive in~$\mathcal{CFK}_\mathrm{alg}$.

Any positive torus knot~$T_{p,q}$, where~$p$ and~$q$ are relatively prime positive integers,
has a\linebreak \emph{semigroup}~$S(T_{p,q})=\langle p,q\rangle:=\{px+qy\mid x,y\in\mathbb{Z}_{\geqslant0}\}$.
When~$T_{p,q}$ is nontrivial,~$\mathit{CFK}^\infty(T_{p,q})$ is a staircase complex~$\mathrm{St}(b_1,\cdots,b_{2m})$
and the semigroup determines~$b_1,\cdots,b_{2m}$ by
\begin{equation}\label{Sb}
\begin{split}
b_1&=1,\\
\text{and inductively}&\\
b_i&=\min\{a>0\mid b_1+\cdots+b_{i-1}+a\in S(T_{p,q})\}\text{ for even }i,\\
b_i&=\min\{a>0\mid b_1+\cdots+b_{i-1}+a\not\in S(T_{p,q})\}\text{ for odd }i;\\
m&=\max\{i>0\mid b_1,\cdots,b_i\text{ are well defined}\}.
\end{split}
\end{equation}

A result of~\cite[Theorem 1.2]{genus} states\begin{equation}\label{Lgenus}b_1+\cdots+b_m=b_{m+1}+\cdots+b_{2m}=\frac{(p-1)(q-1)}{2}.\end{equation}

\subsection{Some facts in~$\mathcal{CFK}_\mathrm{alg}$}

We will use the results of~\cite[Lemmas 4.7, 6.3, 6.4]{ordered} as listed in the next lemma.

\begin{Lem}\label{dominate}The following statements are true.
\begin{enumerate}
\item\label{domination implies independence} If~$0<g_1\ll g_2\ll g_3\ll\cdots$ in a totally ordered abelian group,
then~$g_1,g_2,g_3,\cdots$ are linearly independent.
\item\label{a1greater} If positive integers~$a_1,\cdots,a_m$ and~$b_1,\cdots,b_n$ satisfy~$a_1>b_1$,
then $[a_1,\cdots,a_m]\ll[b_1,\cdots,b_n]$.
\item\label{a2greater} If positive integers~$a_1,\cdots,a_m$ and~$b_1,\cdots,b_n$ satisfy~$a_1=b_1$ and~$a_2<b_2$,\linebreak
then~$[a_1,\cdots,a_m]\ll[b_1,\cdots,b_n]$.
\end{enumerate}\end{Lem}

We will also use~\cite[Lemmas 3.1, 3.2, 4.2]{filtration}, as stated next.

\begin{Lem}\label{split}The following statements are true.
\begin{enumerate}
\item\label{squeeze} Let $a_i,b_j$ be positive integers for $i=1,\cdots,m$ and $j=1,\cdots,n$.
If~$m$ is even\linebreak and $\max\{a_i\mid i\text{ is odd}\}\leqslant b_j\leqslant\min\{a_i\mid i\text{ is even}\}$ for all $j=1,\cdots,n$,
then $$[a_1,\cdots,a_m]+[b_1,\cdots,b_n]=[a_1,\cdots,a_m,b_1,\cdots,b_n].$$
\item\label{riffle} Let $a,b,c_1,\cdots,c_k$ be positive integers and $m,n_1,\cdots,n_k$ be nonnegative integers.\linebreak
If~$a\leqslant c_j\leqslant b$ for all $j=1,\cdots,k$ and $m\leqslant\min\{n_1,\cdots,n_k\}$,
then $$[(1,a)^m,1,b]+[(1,a)^{n_1},1,c_1,\cdots,(1,a)^{n_k},1,c_k]=[(1,a)^m,1,b,(1,a)^{n_1},1,c_1,\cdots,(1,a)^{n_k},1,c_k].$$
\item\label{shorter} Let $a,b,c$ be positive integers and $m,n$ be nonnegative integers.
If~$a<b$,~$a\leqslant c$ and $m<n$ or if~$a\leqslant c<b$ and $m=n$,
then~$[(1,a)^n,1,c]\ll[(1,a)^m,1,b]$.
\end{enumerate}\end{Lem}

\emph{Remark.} When~$m=0$, the notation~$[(1,a)^m,1,b]$ means~$[1,b]$
and the conclusions in~\ref{riffle} and~\ref{shorter} can be deduced from~\ref{squeeze} and Lemma~\ref{dominate}~\ref{a2greater}.

\begin{Ex}\label{121n}Taking~$a=2$, Lemma~\ref{dominate}~\ref{shorter} gives
\begin{align*}
\ll&\cdots\ll\cdots\\
\ll&[(1,2)^3,1,3]\ll[(1,2)^3,1,4]\ll[(1,2)^3,1,5]\ll\cdots\\
\ll&[(1,2)^2,1,3]\ll[(1,2)^2,1,4]\ll[(1,2)^2,1,5]\ll\cdots\\
\ll&[1,2,1,3]\ll[1,2,1,4]\ll[1,2,1,5]\ll\cdots\\
\ll&[1,3]\ll[1,4]\ll[1,5]\ll\cdots
\end{align*}
Moreover, Lemma~\ref{dominate}~\ref{riffle} implies that the sum of any (finitely many) elements above is the concatenation in the order from the greatest to the least ones.
For example,$$[(1,2)^2,1,3]+[1,9]+[1,2,1,6]=[1,9,1,2,1,6,(1,2)^2,1,3].$$
\end{Ex}

\section{Computations}

For any nonnegative integer~$n$, let $$J_n=(T_{6n+5,6n+8}\#-T_{6n+2,6n+5})\#-(T_{6n+4,6n+7}\#-T_{6n+1,6n+4}),$$
$$K_n=(T_{6n+5,6n+7}\#-T_{6n+3,6n+5})\#-(T_{6n+3,6n+5}\#-T_{6n+1,6n+3})$$
and $$L_n=T_{6n+4,6n+5}\#-2T_{6n+3,6n+4}\#T_{6n+2,6n+3}.$$
We will show that the knot~$J_n\#-K_n\#L_n$ has vanishing~$\Upsilon$-invariant and~$\varphi$-invariant but\linebreak nonvanishing~$\varepsilon$-invariant when~$n>0$.

\subsection{The $\Upsilon$-invariant and~$\varphi$-invariant vanish}

Applying Lemma~\ref{Upsilon recursive} with~$k=1$ and applying it again with~$r=2$ or~$3$, it is easy to see the following equalities.

\begin{Lem}\label{mod2mod3Upsilon}For any integer~$n\geqslant0$, we have$$\Upsilon_{T_{3n+1,3n+4}}(t)=n\Upsilon_{T_{3,4}}(t)+\Upsilon_{T_{3n+1,3n+2}}(t)$$
$$\Upsilon_{T_{3n+2,3n+5}}(t)=\Upsilon_{T_{2,3}}(t)+n\Upsilon_{T_{3,4}}(t)+\Upsilon_{T_{3n+2,3n+3}}(t),$$
and $$\Upsilon_{T_{2n+1,2n+3}}(t)=n\Upsilon_{T_{2,3}}(t)+\Upsilon_{T_{2n+1,2n+2}}(t).$$\end{Lem}

Applying Lemma~\ref{phi recursive} similarly yields the following equalities.

\begin{Lem}\label{mod2mod3phi}For any integer~$n>0$, we have$$\varphi(T_{3n+1,3n+4})=2n\varphi(T_{3,4})+\varphi(T_{3n-2,3n+1})+\varphi(T_{3n+1,3n+2})-\varphi(T_{3n,3n+1}),$$
$$\varphi(T_{3n+2,3n+5})=2\varphi(T_{2,3})+2n\varphi(T_{3,4})+\varphi(T_{3n-1,3n+2})+\varphi(T_{3n+2,3n+3})-\varphi(T_{3n+1,3n+2})$$
and $$\varphi(T_{2n+1,2n+3})=2n\varphi(T_{2,3})+\varphi(T_{2n-1,2n+1})+\varphi(T_{2n+1,2n+2})-\varphi(T_{2n,2n+1}).$$\end{Lem}

Now we compute the $\Upsilon$-invariant and~$\varphi$-invariant of~$J_n$ and~$K_n$.

\begin{Lem}\label{vanish}For any integer~$n\geqslant0$, we have $\Upsilon_{J_n\#-K_n\#L_n}(t)=0$ and $\varphi(J_n\#-K_n\#L_n)=0$.\end{Lem}

\begin{proof}[\emph{\bfseries Proof.}]Using Lemma~\ref{mod2mod3Upsilon}, we know
\begin{align*}
\Upsilon_{J_n}(t)=&(\Upsilon_{T_{3(2n+1)+2,3(2n+1)+5}}(t)-\Upsilon_{T_{3(2n)+2,3(2n)+5}}(t))\\
&-(\Upsilon_{T_{3(2n+1)+1,3(2n+1)+4}}(t)-\Upsilon_{T_{3(2n)+1,3(2n)+4}}(t))\\
=&(\Upsilon_{T_{3,4}}(t)+\Upsilon_{T_{3(2n+1)+2,3(2n+1)+3}}(t)-\Upsilon_{T_{3(2n)+2,3(2n)+3}}(t))\\
&-(\Upsilon_{T_{3,4}}(t)+\Upsilon_{T_{3(2n+1)+1,3(2n+1)+2}}(t)-\Upsilon_{T_{3(2n)+1,3(2n)+2}}(t))\\
=&(\Upsilon_{T_{6n+5,6n+6}}(t)-\Upsilon_{T_{6n+2,6n+3}}(t))-(\Upsilon_{T_{6n+4,6n+5}}(t)-\Upsilon_{T_{6n+1,6n+2}}(t))\end{align*}
and
\begin{align*}
\Upsilon_{K_n}(t)=&(\Upsilon_{T_{2(3n+2)+1,2(3n+2)+3}}(t)-\Upsilon_{T_{2(3n+1)+1,2(3n+1)+3}}(t))\\
&-(\Upsilon_{T_{2(3n+1)+1,2(3n+1)+3}}(t)-\Upsilon_{T_{2(3n)+1,2(3n)+3}}(t))\\
=&(\Upsilon_{T_{2,3}}(t)+\Upsilon_{T_{2(3n+2)+1,2(3n+2)+2}}(t)-\Upsilon_{T_{2(3n+1)+1,2(3n+1)+2}}(t))\\
&-(\Upsilon_{T_{2,3}}(t)+\Upsilon_{T_{2(3n+1)+1,2(3n+1)+2}}(t)-\Upsilon_{T_{2(3n)+1,2(3n)+2}}(t))\\
=&\Upsilon_{T_{6n+5,6n+6}}(t)-2\Upsilon_{T_{6n+3,6n+4}}(t)+\Upsilon_{T_{6n+1,6n+2}}(t).\end{align*}

Hence $\Upsilon_{J_n\#-K_n}(t)=-\Upsilon_{T_{6n+2,6n+3}}(t))+2\Upsilon_{T_{6n+3,6n+4}}(t)-\Upsilon_{T_{6n+4,6n+5}}(t)=-\Upsilon_{L_n}(t)$.

Using Lemma~\ref{mod2mod3phi}, we know
\begin{align*}
\varphi(J_n)=&(\varphi(T_{3(2n+1)+2,3(2n+1)+5})-\varphi(T_{3(2n+1)-1,3(2n+1)+2}))\\
&-(\varphi(T_{3(2n+1)+1,3(2n+1)+4})-\varphi(T_{3(2n+1)-2,3(2n+1)+1}))\\
=&(2\varphi(T_{2,3})+2(2n+1)\varphi(T_{3,4})+\varphi(T_{3(2n+1)+2,3(2n+1)+3})-\varphi(T_{3(2n+1)+1,3(2n+1)+2}))\\
&-(2(2n+1)\varphi(T_{3,4})+\varphi(T_{3(2n+1)+1,3(2n+1)+2})-\varphi(T_{3(2n+1),3(2n+1)+1}))\\
=&(2\varphi(T_{2,3})+\varphi(T_{6n+5,6n+6})-\varphi(T_{6n+4,6n+5}))-(\varphi(T_{6n+4,6n+5})-\varphi(T_{6n+3,6n+4}))\end{align*}
and
\begin{align*}
\varphi(K_n)=&(\varphi(T_{2(3n+2)+1,2(3n+2)+3})-\varphi(T_{2(3n+2)-1,2(3n+2)+1}))\\
&-(\varphi(T_{2(3n+1)+1,2(3n+1)+3})-\varphi(T_{2(3n+1)-1,2(3n+1)+1}))\\
=&(2(3n+2)\varphi(T_{2,3})+\varphi(T_{2(3n+2)+1,2(3n+2)+2})-\varphi(T_{2(3n+2),2(3n+2)+1}))\\
&-(2(3n+1)\varphi(T_{2,3})+\varphi(T_{2(3n+1)+1,2(3n+1)+2})-\varphi(T_{2(3n+1),2(3n+1)+1}))\\
=&(2\varphi(T_{2,3})+\varphi(T_{6n+5,6n+6})-\varphi(T_{6n+4,6n+5}))-(\varphi(T_{6n+3,6n+4})-\varphi(T_{6n+2,6n+3}))\end{align*}

Hence $\varphi(J_n\#-K_n)\\=-(\varphi(T_{6n+4,6n+5})-\varphi(T_{6n+3,6n+4}))+(\varphi(T_{6n+3,6n+4})-\varphi(T_{6n+2,6n+3}))=-\varphi(L_n)$.
\end{proof}

\subsection{The~$\varepsilon$-invariant does not vanish}

The following fact can be verified directly. See~\cite[Equation~(7) and the claim below Equation~(6)]{example2}.

\begin{Lem}\label{segment}Suppose $p$ and $r$ are relatively prime positive integers with~$r<p$.
Then$$\langle p,p+r\rangle\cap\{lp,lp+1,lp+2,\cdots,(l+1)p\}=\{lp,lp+r,\cdots,lp+lr,(l+1)p\}$$for any~$l=0,1,\cdots,\lfloor\frac{p}{r}\rfloor$.\end{Lem}

Denote the~$\varepsilon$-equivalence class of the complex~$\mathit{CFK}^\infty(K)$ by~$[\![K]\!]$ for any knot~$K$.
That is to say,~$[\![K]\!]:=[\mathit{CFK}^\infty(K)]$.

\begin{Lem}\label{mod3epsilon}For any integer~$n>0$, we have$$[\![T_{3n+1,3n+4}]\!]=[1,3n]+\sum_{i=1}^{n-1}[(1,2)^i,1,3n-3i]+[(1,2)^n]+[2,\cdots]$$and
$$[\![T_{3n+2,3n+5}]\!]=[1,3n+1]+\sum_{i=1}^{n-1}[(1,2)^i,1,3n+1-3i]+[(1,2)^n]+[1,1,\cdots].$$\end{Lem}

\begin{proof}[\emph{\bfseries Proof.}]According to Lemma~\ref{segment}, the semigroup $S(T_{3n+1,3n+4})=\langle3n+1,3n+4\rangle$ has initial entries
\begin{align*}&0,3n+1,\\
&(3n+1)+3,2(3n+1),\\
&2(3n+1)+3,2(3n+1)+6,3(3n+1),\\
&\cdots,\\
&n(3n+1)+3,n(3n+1)+6,\cdots,n(3n+1)+3n,(n+1)(3n+1),\\
&(n+1)(3n+1)+3,\cdots\end{align*}
and therefore$$\mathit{CFK}^\infty(T_{3n+1,3n+4})=\mathrm{St}(1,3n,(1,2,1,3n-3),((1,2)^2,1,3n-6),\cdots,((1,2)^{n-1},1,3),(1,2)^n,2,2,\cdots)$$by Equation~\eqref{Sb}.
Note that the sum of the entries before ``$2,2$'' is~$n(3n+1)+3n$, which is less than~$\frac{(3n+1-1)(3n+4-1)}{2}$.
This means the first~$2$ in~``$2,2,\cdots$'' belongs to the first half of all entries due to Equation~\eqref{Lgenus}.
Hence\begin{align*}\mathit{CFK}^\infty(T_{3n+1,3n+4})&\\
=\mathrm{St}(&1,3n,(1,2,1,3n-3),((1,2)^2,1,3n-6),\cdots,((1,2)^{n-1},1,3),(1,2)^n,2,a_1,\cdots,a_m,\\
&a_m,\cdots,a_1,2,(2,1)^n,(3,1,(2,1)^{n-1}),\cdots,(3n-6,1,(2,1)^2),(3n-3,1,2,1),3n,1)\end{align*}
and$$[\![T_{3n+2,3n+5}]\!]=[1,3n,(1,2,1,3n-3),(1,2,(1,2)^2,1,3n-6),\cdots,((1,2)^{n-1},1,3),(1,2)^n,2,a_1,\cdots,a_m]$$for some positive integers~$a_1,\cdots,a_m$.
Here~$m$ is allowed to be~$0$ (meaning $a_1,\cdots,a_m$ is an empty sequence), which happens if~$n=1$.

Observe that there cannot be~$3$ consecutive integers~$b,b+1,b+2$ outside~$S(T_{3n+1,3n+4})$ such that~$b\geqslant n(3n+1)$.
Otherwise,~$c,c+1,c+2$ would be~$3$ consecutive integers outside~$S(T_{3n+1,3n+4})$ such that~$n(3n+1)\leqslant c\leqslant(n+1)(3n+1)$,
where~$c=b-(\lfloor\frac{b}{3n+1}\rfloor-n)(3n+1)$.
Therefore any entry in an even slot after the~``$((1,2)^{n-1},1,3)$'' part of
$$\mathrm{St}(\cdots,(1,2)^{n-1},1,3),(1,2)^n,2,a_1,\cdots,a_m,a_m,\cdots,a_1,\cdots)$$ cannot exceed~$2$.
In particular,~$a_i=1$ or~$2$ for all~$i$.

Now we can decompose~$[\![T_{3n+2,3n+5}]\!]$ as
\begin{align*}&[1,3n,(1,2,1,3n-3),(1,2,(1,2)^2,1,3n-6),\cdots,((1,2)^{n-1},1,3),(1,2)^n,2,a_1,\cdots,a_m]\\
=&[1,3n,(1,2,1,3n-3),(1,2,(1,2)^2,1,3n-6),\cdots,((1,2)^{n-1},1,3),(1,2)^n]+[2,a_1,\cdots,a_m]\\
=&[1,3n,(1,2,1,3n-3),(1,2,(1,2)^2,1,3n-6),\cdots,((1,2)^{n-1},1,3)]+[(1,2)^n]+[2,a_1,\cdots,a_m]\\
=&[1,3n]+[(1,2,1,3n-3),(1,2,(1,2)^2,1,3n-6),\cdots,((1,2)^{n-1},1,3)]+[(1,2)^n]+[2,a_1,\cdots,a_m]\end{align*}
by Lemma~\ref{split}~\ref{squeeze}, and the second summand
\begin{align*}&[(1,2,1,3n-3),(1,2,(1,2)^2,1,3n-6),\cdots,((1,2)^{n-1},1,3)]\\
=&[(1,2,1,3n-3)]+[(1,2,(1,2)^2,1,3n-6),\cdots,((1,2)^{n-1},1,3)]\\
=&\cdots\\
=&\sum_{i=1}^{n-1}[(1,2)^i,1,3n-3i]\end{align*}
by using Lemma~\ref{split}~\ref{riffle} repeatedly.
This completes the proof of the first part.

According to Lemma~\ref{segment}, the semigroup $S(T_{3n+2,3n+5})=\langle3n+2,3n+5\rangle$ has initial entries
\begin{align*}&0,3n+2,\\
&(3n+2)+3,2(3n+2),\\
&2(3n+2)+3,2(3n+2)+6,3(3n+2),\\
&\cdots,\\
&n(3n+2)+3,n(3n+2)+6,\cdots,n(3n+2)+3n,(n+1)(3n+2),\\
&(n+1)(3n+2)+3,\cdots\end{align*}
and therefore $\mathit{CFK}^\infty(T_{3n+2,3n+5})\\=\mathrm{St}(1,3n+1,(1,2,1,3n-2),((1,2)^2,1,2,1,3n-5),\cdots,((1,2)^{n-1},1,4),(1,2)^n,1,1,\cdots)$ by\linebreak Equation~\eqref{Sb}.
Note that the sum of the entries before ``$1,1$'' is~$n(3n+2)+3n$, which is less than~$\frac{(3n+2-1)(3n+5-1)}{2}-1$.
This means the first two~$1$'s in~``$1,1,\cdots$'' belong to the first half of all entries due to Equation~\eqref{Lgenus}.
The rest of the proof is similar to the argument for~$T_{3n+1,3n+4}$.
\end{proof}

\begin{Cor}\label{epsilon3terms}For any integer~$m\geqslant7$ not divisible by~$3$, we have$$[\![T_{m,m+3}]\!]=[1,m-1]+[1,2,1,m-4]+O$$with~$0\leqslant O\ll[1,2,1,m-4]$.\end{Cor}

\begin{proof}[\emph{\bfseries Proof.}]The hypothesis implies that~$m=3n+1$ or~$3n+2$ for some integer~$n\geqslant2$.
Thus$$[\![T_{m,m+3}]\!]=[1,m-1]+([1,2,1,m-4]+[(1,2)^2,1,m-7]+\cdots+[(1,2)^{n-1},1,m+2-3n])+[(1,2)^n]+O'$$by Lemma~\ref{mod3epsilon}, where~$O'=[2,\cdots]$ or~$[1,1,\cdots]$.
Example~\ref{121n} shows$$[(1,2)^2,1,m-7]\ll[1,2,1,m-4],\qquad\cdots,\qquad[(1,2)^{n-1},1,m+2-3n]\ll[1,2,1,m-4].$$
We also have~$[(1,2)^n]\ll[1,2,1,m-4]$,~$[2,\cdots]\ll[1,2,1,m-4]$ and~$[1,1,\cdots]\ll[1,2,1,m-4]$
by Lemmas~\ref{split}~\ref{shorter},~\ref{dominate}~\ref{a1greater} and~\ref{dominate}~\ref{a2greater}, respectively.
Taking$$O=([(1,2)^2,1,m-7]+\cdots+[(1,2)^{n-1},1,m+2-3n])+[(1,2)^n]+O',$$the conclusion is proved.
\end{proof}

Similarly but more simply, one can prove the following two lemmas.

\begin{Lem}\label{epsilon2terms2}For any odd integer~$m\geqslant5$, we have~$[\![T_{m,m+2}]\!]=[1,m-1]+O$ with~$0<O\ll[1,2,1,3]$.\end{Lem}

\begin{proof}[\emph{\bfseries Proof.}]According to Lemma~\ref{segment}, the semigroup $S(T_{m,m+2})=\langle m,m+2\rangle$ has initial entries
\begin{align*}&0,m,\\
&m+2,2m,\\
&2m+2,2m+4,3m,\\
&\cdots,\\
&\tfrac{m-1}{2}m+2,\tfrac{m-1}{2}m+4,\tfrac{m-1}{2}m+6,\cdots,\tfrac{m-1}{2}m+(m-1),(\tfrac{m-1}{2}+1)m,\\
&(\tfrac{m-1}{2}+1)m+2,\cdots\end{align*}
and therefore $\mathit{CFK}^\infty(T_{m,m+2})\\=\mathrm{St}(1,m-1,(1,1,1,m-3),((1,1)^2,1,m-5),\cdots,((1,1)^{\frac{m-1}{2}-1},1,2),(1,1)^{\frac{m-1}{2}-1},2,1,\cdots)$
by\linebreak Equation~\eqref{Sb}.
Note that the sum of the entries up to~``$((1,1)^{\frac{m-1}{2}-1},1,2)$'' is~$\frac{m-1}{2}m$, which is~$\frac{m-1}{2}$ less than~$\frac{(m-1)(m+2-1)}{2}$.
By Equation~\eqref{Lgenus}, this implies$$\mathit{CFK}^\infty(T_{m,m+2})=[1,m-1,(1,1,1,m-3),((1,1)^2,1,m-5),\cdots,((1,1)^{\frac{m-1}{2}-1},1,2),(1)^{\frac{m-1}{2}}].$$
Thus$$[\![T_{m,m+2}]\!]=[1,m-1]+[(1,1,1,m-3),((1,1)^2,1,m-5),\cdots,((1,1)^{\frac{m-1}{2}-1},1,2),(1)^{\frac{m-1}{2}}]$$by Lemma~\ref{split}~\ref{squeeze}.
Take~$O$ to be the second summand. Then~$O\ll[1,2,1,3]$ by Lemma~\ref{dominate}~\ref{a2greater}.
\end{proof}

\begin{Lem}\label{epsilon2terms1}For any integer~$m\geqslant3$, we have~$[\![T_{m,m+1}]\!]=[1,m-1]+O$ with~$0\leqslant O\ll[1,2,1,3]$.\end{Lem}

\begin{proof}[\emph{\bfseries Proof.}]According to Lemma~\ref{segment}, the semigroup $S(T_{m,m+1})=\langle m,m+1\rangle$ has initial entries
\begin{align*}&0,m,\\
&m+1,2m,\\
&2m+1,2m+2,3m,\\
&\cdots,\\
&(m-2)m+1,(m-2)m+2,\cdots,(m-2)m+m-2,(m-1)m,\\
&(m-1)m+1,\cdots\end{align*}
and therefore~$\mathit{CFK}^\infty(T_{m,m+1})=\mathrm{St}(1,m-1,2,m-2,\cdots,m-1,1,\cdots)$ by Equation~\eqref{Sb}.
Note that the sum of the entries listed is~$(m-1)m$, which is~$2\frac{(m-1)(m+1-1)}{2}$.
By Equation~\eqref{Lgenus}, this implies$$\mathit{CFK}^\infty(T_{3n+1,3n+4})=[1,m-1,2,\cdots].$$
Here no entries exceed~$m-1$.
Thus$$[\![T_{m,m+1}]\!]=[1,m-1]+[2,\cdots]$$by Lemma~\ref{split}~\ref{squeeze}.
Take~$O$ to be the second summand. Then~$O\ll[1,2,1,3]$ by Lemma~\ref{dominate}~\ref{a1greater}.
\end{proof}

\begin{Prop}\label{nonvanish}For any integer~$n>0$, we have $[\![J_n\#-K_n\#L_n]\!]=[1,2,1,6n+1]+O$\linebreak with~$|O|\ll[1,2,1,6n+1]$.
Here~$|O|:=\left\{\begin{aligned}&O&\text{ if }O\geqslant0,\\&-O&\text{ if }O<0.\end{aligned}\right.$\end{Prop}

\begin{proof}[\emph{\bfseries Proof.}]Using Corollary~\ref{epsilon3terms}, Lemma~\ref{epsilon2terms2} and Lemma~\ref{epsilon2terms1}, respectively, we know
\begin{align*}
[\![J_n]\!]=&[\![T_{6n+5,6n+8}]\!]-[\![T_{6n+2,6n+5}]\!])\\
&-([\![T_{6n+4,6n+7}]\!]-[\![T_{6n+1,6n+4}]\!])\\
=&([1,6n+4]+[1,2,1,6n+1]+O_1-[1,6n+1]-[1,2,1,6n-2]-O_2)\\
&-([1,6n+3]+[1,2,1,6n]+O_3-[1,6n]-[1,2,1,6n-3]-O_4),\end{align*}
\begin{align*}
[\![K_n]\!]=&([\![T_{6n+5,6n+7}]\!]-[\![T_{6n+3,6n+5}]\!])\\
&-([\![T_{6n+3,6n+5}]\!]-[\![T_{6n+1,6n+3}]\!])\\
=&([1,6n+4]+O_5-[1,6n+2]-O_6)\\
&-([1,6n+2]+O_7-[1,6n]-O_8)\end{align*}
and
\begin{align*}
[\![L_n]\!]=&[\![T_{6n+4,6n+5}]\!]-2[\![T_{6n+3,6n+4}]\!]+[\![T_{6n+2,6n+3}]\!]\\
=&([1,6n+3]+O_9)-2([1,6n+2]+O_{10})+([1,6n+1]+O_{11})\end{align*}
for some~$O_1,\cdots,O_{11}$ with
\begin{align*}&0\leqslant O_1\ll[1,2,1,6n+1],\\
&0\leqslant O_2\ll[1,2,1,6n-2]\ll[1,2,1,6n+1],\\
&0\leqslant O_3\ll[1,2,1,6n]\ll[1,2,1,6n+1],\\
&0\leqslant O_4\ll[1,2,1,6n-3]\ll[1,2,1,6n+1],\\
&O_5,\cdots,O_{11}\ll[1,2,1,3]\ll[1,2,1,6n+1].\end{align*}

Let~$O_{12}=(O_1-O_2-O_3+O_4)-(O_5-O_6-O_7+O_8)+(O_9-2O_{10}+O_{11})$. Then~$|O_{12}|\ll[1,2,1,6n+1]$. Finally
\begin{align*}
[\![J_n\#-K_n\#L_n]\!]=&[1,2,1,6n+1]-[1,2,1,6n-2]-[1,2,1,6n]+[1,2,1,6n-3]+O_{12}\\
=&[1,2,1,6n+1]-([1,2,1,6n]+[1,2,1,6n-2]-[1,2,1,6n-3]-O_{12})\end{align*}
where~$|[1,2,1,6n]+[1,2,1,6n-2]-[1,2,1,6n-3]-O_{12}|\ll[1,2,1,6n+1]$ by Example~\ref{121n}.
\end{proof}

\emph{Remark.} It is easy to verify that~$[\![J_0\#-K_0\#L_0]\!]=0$ by using Lemma~\ref{split}~\ref{squeeze},~\ref{riffle}.

\subsection{Proof of the theorem}

\begin{Thm}\label{main}The family of knots~$\{J_n\#-K_n\#L_n\}_{n=1}^\infty$ generates a subgroup of $\mathcal{C}$ isomorphic to $\mathbb{Z}^\infty$ such that
each of its nonzero elements has vanishing~$\Upsilon$-invariant and~$\varphi$-invariant but nonvanishing $\varepsilon$-invariant.\end{Thm}

\begin{proof}[\emph{\bfseries Proof.}]Obviously any linear combination of knots in the family has vanishing~$\Upsilon$-invariant\linebreak and~$\varphi$-invariant by Lemma~\ref{vanish}.

From Proposition~\ref{nonvanish}, we know~$[\![J_n\#-K_n\#L_n]\!]=[1,2,1,6n+1]+O_n$ with~$|O_n|\ll[1,2,1,6n+1]$.
This implies~$[\![J_n\#-K_n\#L_n]\!]\ll[\![J_{n+1}\#-K_{n+1}\#L_{n+1}]\!]$ for each~$n$.
In fact,
\begin{align*}&[\![J_{n+1}\#-K_{n+1}\#L_{n+1}]\!]-k[\![J_n\#-K_n\#L_n]\!]\\
=&([1,2,1,6(n+1)+1]-O_{n+1})-k([1,2,1,6n+1]-O_n)\\
=&[1,2,1,6(n+1)+1]-(O_{n+1}+k[1,2,1,6n+1]-kO_n)\\
>&0.\end{align*}
for any natural number~$k$.

Now~$\{[\![J_n\#-K_n\#L_n]\!]\}_{n=1}^\infty$ is a basis of a free subgroup of infinite rank in~$\mathcal{CFK}_\mathrm{alg}$ by Lemma~\ref{dominate}~\ref{domination implies independence}.\linebreak
Therefore so is~$\{J_n\#-K_n\#L_n\}_{n=1}^\infty$ in~$\mathcal{C}$.
Any nontrivial linear combination of this family has nonvanishing~$\varepsilon$-invariant, because it maps to a nontrivial linear combination of~$\{[\![J_n\#-K_n\#L_n]\!]\}_{n=1}^\infty$
under the homomorphism from~$\mathcal{C}$ to~$\mathcal{CFK}_\mathrm{alg}$.
\end{proof}

\end{document}